\documentclass[12pt, reqno]{amsart}
\usepackage[margin=1in]{geometry}
\usepackage{bbm}
\usepackage{graphicx}
\usepackage{amssymb}
\usepackage{amsmath}
\usepackage{amsfonts}
\usepackage{amstext}
\usepackage{amsthm}
\usepackage{setspace}
\usepackage{array}
\usepackage{textcomp}
\usepackage{tipa}

\usepackage{graphicx}
\usepackage[cp1251]{inputenc}
\usepackage{wrapfig}
\usepackage{caption}
\usepackage{subcaption}

\newcommand{\be}{\begin{equation} }
\newcommand{\ee}{\end{equation} }
\newcommand{\bee}{\begin{eqnarray} }
\newcommand{\eee}{\end{eqnarray} }

\numberwithin{equation}{section}

\newtheorem{theorem}{Theorem}

\newtheorem{remark}{Remark}

\newtheorem{proposition}{Proposition}

\newcommand{\EE}[1]{\mathbb E}

\usepackage{xcolor}
\usepackage[normalem]{ulem}
\usepackage{hyperref}

\hypersetup{colorlinks=true,pdfborder=001,  citecolor  = blue, citebordercolor= {magenta}}
\hypersetup{linkcolor=magenta, linkbordercolor = blue}
\hypersetup{colorlinks,urlcolor=blue}

\newcommand{\E}{\mathbb E}

\begin{document}
\title{On the time constant of high dimensional first passage percolation, revisited} 

\author{Antonio Auffinger}
\address{Department of Mathematics, Northwestern University,  United States of America}
\email{tuca@northwestern.edu}

\author{ Si~Tang}
\address{Department of Mathematics, Lehigh University, United States of America} 
\email{sit218@lehigh.edu}


\maketitle



\begin{abstract}In  \cite{AT16highdimenFPP}, it was claimed that the time constant $\mu_{d}(e_{1})$ for the first-passage percolation model on $\mathbb Z^{d}$ is $\mu_{d}(e_{1}) \sim \log d/(2ad)$ as $d\to \infty$, if the passage times $(\tau_{e})_{e\in \mathbb E^{d}}$ are i.i.d., with a common c.d.f. $F$ satisfying 
$\left|\frac{F(x)}{x}-a\right| \le \frac{C}{|\log x|}$ for some constants $a, C$ and sufficiently small $x$. 

However, the proof of the upper bound, namely, Equation (2.1) in \cite{AT16highdimenFPP}
\begin{align}\label{eq:abstract}
\limsup_{d\to\infty} \frac{\mu_{d}(e_{1})ad}{\log d} \le \frac{1}{2}
\end{align}
is incorrect. In this article, we provide a different approach that establishes \eqref{eq:abstract}. As a side product of this new method, we also show that the variance of the non-backtracking passage time to the first hyperplane is of order $o\big((\log d/d)^{2}\big)$ as $d\to \infty$ in the case of the when the edge weights are exponentially distributed.
\end{abstract}

\section{Introduction}

In this paper we study first passage percolation on $\mathbb Z^{d}$ which is defined as follows. At each nearest-neighbor edge in $\mathbb Z^d$, we attach a non-negative random variable $\tau_e$, known as the passage time of the edge $e$.  These random variables $(\tau_e)$ are independent and identically distributed with common distribution $F$. We will also assume that $F$ satisfies the following conditions: {for some $a, C$ and $\epsilon_0>0$,
\begin{align}
\label{eqn:conditionat0}
&\left|\frac{F(x)}{x}-a\right|\le C \cdot |\log x |^{-1}, \text{ for all } x \in [0, \epsilon_0]\\
\label{eqn:finite-expectation}
\text{ and }&\int x dF(x) < \infty.
\end{align}

 A {\it path} $\gamma$ is a sequence of nearest-neighbor edges   in $\mathbb Z^d$ such the starting point of each edge coincides with the endpoint of the previous edge. For any finite path $\gamma$ we define the passage time of $\gamma$ to be
\[
T(\gamma)=\sum_{e \in \gamma} \tau_e.
\]
Given two points $x,y \in \mathbb{Z}^d$ we define
\begin{equation*}\label{definition:passagetime}
T(x,y) = \inf_{\gamma} T(\gamma),
\end{equation*}
where the infimum is over all finite paths $\gamma$ that start at the point $x$ and end at $y$. For a review, we invite the readers to see the book \cite{FPPSurvey} or the classical paper of Kesten \cite{KES86AFPP}.

If $\E \tau_{e} < \infty$, the following limit exists 
\begin{equation*}
 \mu_d(e_1) := \lim_{n \rightarrow \infty} \frac{T(0,ne_1)}{n} \quad \text{ a.s. and in } L^1.
\end{equation*}
and $\mu_d(e_1)$ is called the time constant.
 See \cite[Theorem 2.1]{FPPSurvey} and the discussion therein. }

In this paper, we prove the following limit for $\mu_{d}(e_{1})$, as the dimension $d\to\infty$.
\begin{theorem} \label{thm1}Assume \eqref{eqn:conditionat0} and \eqref{eqn:finite-expectation}. Then
\[
\lim_{d\to \infty} \frac{\mu_{d}(e_{1})d}{\log d} = \frac{1}{2a}.
\]
\end{theorem}

This result was first claimed in  \cite{AT16highdimenFPP}. However, the proof for the upper bound there, namely,
\begin{align*}
\limsup_{d\to\infty} \frac{\mu_{d}(e_{1})ad}{\log d}  \le \frac{1}{2}
\end{align*}
contains an error. Specifically, for the choices of $p, n$ and $x$ in \cite[Equation (3.1)]{AT16highdimenFPP}, the error term is not negligible compared to the main order $(\log d/ d)$ as $d\to\infty$. Here, we fix this error by presenting a new method that also has consequences to point-to-hyperplane passage times.

The main result in this paper is the following.

\begin{theorem}\label{thm:mainUB} Assume \eqref{eqn:conditionat0} and \eqref{eqn:finite-expectation}. The following bound holds:
    \begin{align*}
\limsup_{d\to\infty} \frac{\mu_{d}(e_{1})d}{\log d}  \le \frac{1}{2a}.
\end{align*}
\end{theorem}

The lower bound  was correctly established in  \cite{AT16highdimenFPP}, which we state below. 
\begin{proposition}[{\cite[Proposition 4.1]{AT16highdimenFPP}}] \label{thm-lowerbd}Assume $F(x)/x\to a$ as $x\to 0$ and \eqref{eqn:finite-expectation} for the passage times,  then
\[
\liminf_{d\to \infty} \frac{\mu_{d}(e_{1})ad}{\log d} \ge \frac{1}{2}.
\]
\end{proposition}
\begin{remark} Unlike the upper bound, the proof of the lower bound does not require a specific rate of convergence as in \eqref{eqn:conditionat0}, as long as $F(x)/x\to a$ when sending $x\to 0$.
\end{remark}

\begin{proof}[Proof of Theorem \ref{thm1}]
It follows from the combination of Theorem \ref{thm:mainUB} and Proposition \ref{thm-lowerbd}.
\end{proof}

As an outcome of the new method we introduce to prove Theorem \ref{thm:mainUB}, we also obtain the following theorem. Let 
$$\tilde s_{0,1} := \inf\left\{T(\gamma): 
\begin{array}{l}
\gamma: \mathbf 0 \to \mathbb H_{1} \text{ such that except for the end point, }\\
\text{all other vertices on }\gamma \text{ are contained in }\mathbb H_{0}
\end{array}
\right\},$$
where $\mathbb H_{n}:=\{(x_{1}, \ldots, x_{d})\in \mathbb Z^{d}: x_{1}=n\}$ is the $n$-th hyperplane orthogonal to $e_{1}$.
\begin{theorem} \label{thm-convPL1}Assume \eqref{eqn:conditionat0} and \eqref{eqn:finite-expectation}. Then, as $d\to\infty$
\[
\frac{2ad \tilde s_{0,1}}{\log d} \to 1 \text{ in probability and in $L^1$}.
\]
\end{theorem}

Let us make a few comments about our strategy and how the rest of the paper is organized. Our approach will first consider (the edge version of) the Eden model \cite{Eden}, i.e., the FPP model where  the passage times are  i.i.d. Exponential$(a)$.  Under the Exponential setting, we will make use of Dhar's exploration idea (see \cite{DHA88FPPManyDimen}), and then combine with an appropriate coupling between the $F$-distribution and the Exponential$(a)$ distribution (a similar coupling was also used in \cite[Section 6]{MAR18FPPcartessian}).

More precise, Dhar used in \cite{DHA88FPPManyDimen} a cluster exploration process to predict that for the Exponential$(a)$ FFP model
\begin{equation}\label{eq:Dharpre}
 \limsup_{d\to\infty}\frac{2ad}{\log d}\E \tilde s_{0,1}\le 1.
\end{equation}
By a  standard subadditivity argument (see, e.g., \cite[Theorem 4.2.5]{HW65}, \cite[pp.246]{KES86AFPP} or \cite[Lemma 5.2]{SW78}), we know that $\mu_{d}(e_{1})$ is no larger than $ \E \tilde s_{0,1}$. So Equation \eqref{eq:Dharpre} implies an upper bound for $\mu_{d}(e_{1})$ in the Exponential$(a)$ case. Proposition \ref{thm-lowerbd} then implies that 
\begin{equation*}
\E \frac{2ad}{\log d}\tilde s_{0,1}\to 1.
\end{equation*}

The main obstacle that prevents us from directly generalizing this result to other distributions is that the exploration process will only provide convergence in expectation for $\tilde s_{0,1}$, which is not enough for our purposes. Thus, we first establish a convergence-in-probability result $\frac{2ad}{\log d}\tilde s_{0,1}\xrightarrow{P} 1$ for the Exponential$(a)$-weighted case by showing
\[
\text{Var}(\tilde s_{0,1}) = o \left(\bigg(\frac{\log d}{d}\bigg)^{2}\right ), \quad \text{as } d\to \infty,
\]
expanding on Dhar's cluster exploration idea (see Sections \ref{sec:1st} and \ref{sec:2nd}). 

In Section 4, we derive a coupling between $F$-distributed FPP and the Eden model that preserves the convergence in probability.  Finally, in Section \ref{sec:ui}, we show that the collection of random variables $\{(2ad/\log d )\, \tilde s_{0,1}\}_{d\ge 1}$ is uniformly integrable, which completes a proof of the $L_1$-convergence of $\frac{2ad}{\log d}\tilde s_{0,1}\to 1$ in the $F$-distributed case. 

\subsection{Acknowledgements} A.A's research is partially supported by Simons Foundation SFARI (597491-RWC), NSF Grant 2154076, and a Simons Fellowship. S.T.'s research is partially supported by the Collaboration Grant (\#712728) from the Simons Foundation and the LEAPS-MPS Award (DMS-2137614) from NSF. Both authors thank Wai-Kit Lam for pointing out a simple argument for the uniform integrability of $\{(2ad/\log d )\, \tilde s_{0,1}\}_{d\ge 1}$ and for reading an earlier draft. Both authors thank an anonymous referee for many useful comments that improved the presentation of this paper.

\section{\label{sec:1st}A recap of Dhar's exploration idea in the Exponential case}
In this section, we  recap Dhar's cluster exploration idea, where we introduce the notations and fill in some technical details that were omitted in \cite{DHA88FPPManyDimen}.  Without loss of generality, we assume $a=1$ and consider the standard Exponential case. Note that the exploration idea works nicely because in this case, the first-passage percolation model is Markovian: given the configuration of already-infected sites at any time $t$, the time until next infection is independent of the history before time $t$.  

Consider any infected cluster $C\subset \mathbb H_{0}$ that contains $i$ vertices and $S$ perimeter edges within $\mathbb H_{0}$. Let $T(C)$ denote the waiting time until infection reaches $\mathbb H_{1}$ using a non-backtracking path (i.e., one of those paths in the definition of $\tilde s_{0,1}$). There are $i+S$ possible edges to cross for the next infection to happen, where $i$ of them are along the $e_{1}$ direction and leading to $\mathbb H_{1}$ (denoted by $f_{1}, f_{2}, \ldots, f_{i}$) and $S$ of them remain in the hyperplane $\mathbb H_{0}$ (denoted by $f_{i+1}, \ldots, f_{i+S}$). The passage times of these $i+S$ edges are i.i.d. Exponential(1) random variables. Let $Y$ denote the edge being crossed when next infection occurs (i.e., one of the boundary edges of $C$ with the smallest passage time), and we have
\begin{align*}
\E [T(C)] = \E \bigg[\E [T(C)|Y]\bigg] &= \sum_{j=1}^{i+S} \E [T(C)|Y=f_{j}] P(Y=f_{j}) \\
&= \sum_{j=1}^{i+S} \E [T(C)|Y=f_{j}]  \cdot\frac{1}{i+S}.
\end{align*}
Given that $Y= f_{1}, \ldots, f_{i}$, then the conditional distribution of $T(C)$ is the same as the minimum of $(i+S)$ i.i.d. Exponential(1) random variables and thus
$$\E [T(C)|Y=f_{j}] = \frac{1}{i+S} \quad \text{for }j=1,2,\ldots, i.$$
Given that $Y=f_{i+1}, \ldots, f_{i+S}$, then the next infection will cross a perimeter edge of $C$ within $\mathbb H_{0}$ and expand the infected cluster to a larger cluster $C'$ of $(i+1)$ vertices. In this case, the conditional distribution of $T(C)$ is the same as 
\[
\text{Exponential}(i+S) + T(C\cup f_{j})
\]
where the Exponential random variable is independent of $T(C\cup f_{j})$, due to the Markov property. In this case, we have
$$\E [T(C)|Y=f_{j}] = \frac{1}{i+S} + \E [T(C\cup f_{j})] \quad \text{for }j=i+1, i+2,\ldots, i+S.$$
Combining these together, we obtain
\begin{align*}
\E [T(C)] &=\left\{\frac{i}{i+S} + \sum_{j=i+1}^{i+S}  \bigg[\frac{1}{i+S} + \E [T(C\cup f_{j})]\bigg]\right\} \cdot  \frac{1}{i+S}\\
&=\frac{1}{i+S}\bigg[1 +\sum_{C'}  \E [T(C')]\bigg],
\end{align*}
where the sum is over all clusters $C'\subset \mathbb H_{0}$ of $(i+1)$ infected vertices that can be obtained from cluster $C$ by infecting one additional healthy vertex adjacent to $C$. For each $i=1,2,3, \ldots$, define
\[
x_{i}:=\max_{C: C\subset \mathbb H_{0}, |C|=i} \E T(C).
\]
Note that $x_1=\E \tilde s_{0,1}$ is the quantity of interest here. Taking the maximum over all clusters $C'$, we obtain
\begin{align}
\label{eqn:iter-1pre}
\E[T(C)] \le \frac{1}{i+S}[1+Sx_{i+1}] = \frac{1}{i+S} + \frac{1}{i/S+1}x_{i+1}.
\end{align}
For any cluster $C\subset \mathbb H_{0}$ of $i$ vertices, the number of its perimeter edges in $\mathbb H_{0}$ is bounded above by $i$ times $2(d-1)$, the maximum number of edges in $\mathbb H_{0}$ adjacent to a given vertex. Meanwhile, by the edge-isoperimetric inequality on $\mathbb Z^d$ \cite[Theorem 8]{BL}, one has for any set in a box $\{0, \ldots, \ell\}^{d-1} \subset \mathbb Z^{d-1}$ of cardinality $i \leq \ell^{d-1}/2$  the number of perimeter edges is bounded below by
\begin{equation}\label{EBineq}
\min_{1\leq k \leq d-1} \left \{ 2k i^{1-\frac{1}{k}}\ell^{\frac{d-1}{k}-1}\right \}.
\end{equation}
Choosing $\ell = i$, with $i>3$, the minimum of \eqref{EBineq} is achieved at $k=d-1$ and thus we obtain the bounds


\begin{align}\label{eq:edgeboundary}
s_{i}:=2(d-1)i^{\frac{d-2}{d-1}}\le S\le (2d-1)i.
\end{align}
Here, we also used the fact that the lower bound in \eqref{eq:edgeboundary} for $i=1,2,3$ can be easily checked case by case as follows:
\begin{align*}
i=1, &\ \  \ S =2(d-1) \ge s_1,\\
i=2, &\ \  \ S =4(d-1)-2 \ge s_2, \\
i=3, &\ \ \ S =6(d-1)-4 \ge s_3,
\end{align*}
as long as $d$ is large enough.
Therefore, using \eqref{eq:edgeboundary}, we have

\begin{align}
\label{eqn:geometry}
\frac{1}{i+S} \le \frac{1}{i+s_{i}}\quad \text{and}\quad \frac{1}{i/S+1}\le \frac{1}{1/(2d-1)+1} = 1-\frac{1}{2d}.
\end{align}
Thus, we bound the right side of \eqref{eqn:iter-1pre} by
\[
\E[T(C)] \le \frac{1}{i+s_{i}}+\left(1-\frac{1}{2d}\right)x_{i+1}.
\]
Write $A=1-\frac{1}{2d}$ for notation simplicity. Now taking the maximum over all possible cluster $C\subset \mathbb H_{0}$ with $i$ vertices yields an iterative inequality
\begin{align}
\label{eqn:iter-x}
x_{i}\le \frac{1}{i+s_{i}}+Ax_{i+1}.
\end{align}
 Iterating the relation, we get an upper bound for $\E \tilde s_{0,1}$
\begin{align}
\notag
\E \tilde s_{0,1} &= x_{1} \le \frac{1}{1+s_{1}} + Ax_{2}\le \frac{1}{1+s_{1}} + A\left[\frac{1}{2+s_{2}}+Ax_{3}\right] \\
\label{eqn:1st-sum-bound}
&\le \cdots \le \frac{1}{1+s_{1}} +\frac{A}{2+s_{2}} + \frac{A^{2}}{3+s_{3}}+\cdots \le \frac{1}{1+s_{1}}  + \sum_{n=2}^{\infty}\frac{A^{n-1}}{s_{n}}.
\end{align}
The reason to single out the first term in the summation is to avoid integrating from 0 when bounding the infinite sum by an integral in the next step. Plugging in the expression for $s_{i}$, we get 
\begin{align}
\notag
\E \tilde s_{0,1} 
&\le  \frac{1}{2d-1}  +\frac{1}{2(d-1)}\sum_{n=2}^{\infty}A^{n-1}\frac{1}{n^{\frac{d-2}{d-1}}}\\
\label{eqn:upbd-s01}
&\le \frac{1}{2d-1} + \frac{1}{2(d-1)}\int_{1}^{\infty}A^{x-1}\frac{1}{x^{\frac{d-2}{d-1}}} dx.
\end{align}
We break the integral above into two parts: $\int_{1}^{2d}$ and $\int_{2d}^{\infty}$. For the first integral, we have
\begin{align*}
&\ \ \  \int_{1}^{2d}A^{x-1}\frac{1}{x^{\frac{d-2}{d-1}}} dx  =\int_{1}^{2d}A^{x-1}\frac{x^{\frac{1}{d-1}}}{x} dx\le \frac{(2d)^{\frac{1}{d-1}}}{A}\int_{1}^{\infty}\frac{A^{x}}{x} dx. 
\end{align*}
Notice that as $d\to \infty$, $(2d)^{\frac{1}{d-1}}\to 1$, $A\to 1$ and 
$\int_{1}^{\infty}\frac{A^{x}}{x} dx \sim \log d.$
To see the last asymptotic, we apply L'Hopital's rule and get
\begin{align}
\notag
& \lim_{d\to\infty}\frac{\int_{1}^{\infty}A^{x}\frac{1}{x} dx}{\log d} = \lim_{d\to\infty}\frac{\frac{1}{2d^{2}A}\int_{1}^{\infty}A^{x} dx}{1/d}\\
\label{eqn:lhopital-1}
&=\lim_{d\to\infty} \frac{1}{2dA}\cdot \frac{\left.(1-\frac{1}{2d})^{x}\right |_{x=1}^{\infty}}{\log (1-1/(2d))} = \lim_{d\to\infty} \frac{1}{2d}\cdot \frac{\frac{1}{2d}-1}{-\frac{1}{2d}}=1.
\end{align}
For the second integral, we get
\begin{align*}
 \int_{2d}^{\infty}\left(1-\frac{1}{2d}\right)^{x-1}\frac{1}{x^{\frac{d-2}{d-1}}} dx& \le \frac{1}{(2d)^{\frac{d-2}{d-1}}}  \int_{2d}^{\infty}\left(1-\frac{1}{2d}\right)^{x-1}dx \\
&= \frac{(2d)^{\frac{1}{d-1}}}{2d} \cdot \frac{\left.(1-\frac{1}{2d})^{x}\right |_{x=2d}^{\infty}}{\log (1-1/(2d))}\\
&= (2d)^{\frac{1}{d-1}}\cdot \frac{-(1-\frac{1}{2d})^{2d}}{2d\log (1-1/(2d))} \to e^{-1}.
\end{align*}
Thus, combining the two integrals, we have as $d\to \infty$, 
$$\int_{1}^{\infty}\left(1-\frac{1}{2d}\right)^{x-1}\frac{1}{x^{\frac{d-2}{d-1}}} dx \sim \log d.$$
Plugging this back to \eqref{eqn:upbd-s01}, we prove $\E \tilde s_{0,1}\le \frac{\log d}{2d}(1+o(1))$ for $d$ large. Combining with Proposition \ref{thm-lowerbd}, we conclude 
\begin{equation*}
\E \tilde s_{0,1} \sim \frac{\log d}{2d}
\end{equation*}
as $d\to\infty$ in the i.i.d. Exponentially-weighted case.
\section{\label{sec:2nd}Iteration of the second moment in the Exponential case}
In this section, we prove that $\E \tilde s_{0,1}^{2} \le \left(\frac{\log d}{2d}\right)^{2}(1+o(1))$ as $d\to\infty$ in the Exponential case. Thus, combining with the first moment result, 
$\E \tilde s_{0,1} \sim \frac{\log d}{2d}$, 
we achieve $\displaystyle \text{Var} (\tilde s_{0,1}) = o \left(\bigg(\frac{\log d}{d}\bigg)^{2}\right )$ as desired. Like before, consider any already-infected cluster $C\subset \mathbb H_{0}$ with $i$ infected vertices and $S$ perimeter edges within $\mathbb H_{0}$, and define $T(C)$ and $Y$ in the same way. 
For each $i=1,2,3, \ldots$, define
\[
y_{i}:=\max_{C: C\subset \mathbb H_{0}, |C|=i} \E T^{2}(C). 
\]
We now derive an iterative inequality that relates $y_{i}$ with $y_{i+1}$ (and also $x_{i+1}$, see below). Again, using law of total expectation, we get
\begin{align*}
\E [T^{2}(C)] & = \E \bigg[\E [T^{2}(C)|Y]\bigg]=\sum_{j=1}^{i+S} \E [T^{2}(C)|Y=f_{j}] P(Y=f_{j}) \\
&= \sum_{j=1}^{i+S} \E [T^{2}(C)|Y=f_{j}]  \cdot\frac{1}{i+S} .
\end{align*}
Like before, given $Y= f_{j}$ for $j=1, \ldots, i$, then $T(C)$ follows an Exponential$(i+S)$ distribution and 
$$\E [T^{2}(C)|Y=f_{j}]  =\frac{2}{(i+S)^{2}},\ \ j=1,2,\ldots, i,$$
whereas given $Y=f_{j}$ for $j=i+1,\ldots, i+S$, $T(C)$ has the same distribution as Exponential$(i+S) + T(C\cup f_{j})$ distribution. Thus, recalling that $x_{i}=\max_{C: C\subset \mathbb H_{0}, |C|=i} \E T(C)$, we have
\begin{align*}
\E [T^{2}(C)|Y=f_{j}] &=\frac{2}{(i+S)^{2}} + \E[T^{2}(C\cup f_{j})] + 2 \cdot \frac{1}{i+S}\cdot  \E[T(C\cup f_{j})] \\
&\le \frac{2}{(i+S)^{2}} + y_{i+1} + \frac{2}{i+S} x_{i+1}, \ \ j=i+1,\ldots, i+S.
\end{align*}
Therefore, we get
\begin{align*}
\E [T^{2}(C)] &\le \left\{ \frac{2i}{(i+S)^{2}} +  \frac{2S}{(i+S)^{2}} + Sy_{i+1} + \frac{2S}{i+S} x_{i+1}\right\} \cdot\frac{1}{i+S} \\
&=\frac{2}{(i+S)^{2}}(1+Sx_{i+1}) + \frac{S}{i+S}y_{i+1}\\
&\le \frac{2}{(i+s_{i})^{2}}+\frac{2}{i+s_{i}}\bigg(1-\frac{1}{2d}\bigg)x_{i+1} + \bigg(1-\frac{1}{2d}\bigg)y_{i+1}
\end{align*}
where, in the last step, we use \eqref{eqn:geometry} again to make sure that the right hand side does not depend on the boundary size $S$ of the cluster $C$. Again, we write $A=1-\frac{1}{2d}$. Now taking the maximum over all clusters of size $i$ on the left hand side, we get an iterative inequality
\begin{align}
\label{eqn:iter-y}
y_{i}\le \frac{2}{(i+s_{i})^{2}}+\frac{2A}{i+s_{i}}x_{i+1} + Ay_{i+1}.
\end{align}
Keeping iterating this relation together with \eqref{eqn:iter-x} and noticing that $\E \tilde s_{0,1}^{2} = y_{1}$, we achieve an upper bound for $\E \tilde s_{0,1}^{2}$, which is of the following form
\begin{align*}
\E \tilde s_{0,1}^{2} &=y_{1}\le \frac{2}{(1+s_{1})^{2}}+\frac{2A}{1+s_{1}}x_{2} + Ay_{2}\\
&\le \frac{2}{(1+s_{1})^{2}}+\frac{2A}{1+s_{1}}\left[\frac{1}{2+s_{2}}+Ax_{3}\right] 
 +A\left[\frac{2}{(2+s_{2})^{2}}+\frac{2A}{2+s_{2}}x_{3} + Ay_{3}\right]\\
&=\frac{2}{(1+s_{1})^{2}} + \frac{2A}{2+s_{2}}\left[\frac{1}{1+s_{1}}+\frac{1}{2+s_{2}}\right] + 2 A^{2}\left[\frac{1}{1+s_{1}}+\frac{1}{2+s_{2}}\right] x_{3}+ A^{2}y_{3}\\
&\le \frac{2}{(1+s_{1})^{2}} + \frac{2A}{2+s_{2}}\left[\frac{1}{1+s_{1}}+\frac{1}{2+s_{2}}\right] +\frac{2A^{2}}{3+s_{3}}\left[\frac{1}{1+s_{1}}+\frac{1}{2+s_{2}}+\frac{1}{3+s_{3}}\right] \\
&\qquad + 2A^{3}\left[\frac{1}{1+s_{1}}+\frac{1}{2+s_{2}}+\frac{1}{3+s_{3}}\right] x_{4} + A^{3}y_{4} \le \cdots\\
&\le 2\sum_{n=1}^{\infty}\frac{A^{n-1}}{n+s_{n}}\sum_{k=1}^{n}\frac{1}{k+s_{k}} =2\sum_{k=1}^{\infty}\frac{1}{k+s_{k}}\sum_{n=k}^{\infty}\frac{A^{n-1}}{n+s_{n}}.
\end{align*}
where each inequality follows plugging-in Equations \eqref{eqn:iter-x} and \eqref{eqn:iter-y}, and the equalities are simple expansions of the brackets.
Again, we use integral to bound the double sum from above. First, for each $k$ fixed, $\frac{A^{n-1}}{n+s_{n}}$ decreases in $n$, and thus the inner sum over $n$ is no more than
$$\sum_{n=k}^{\infty}\frac{A^{n-1}}{n+s_{n}} \le \int^{\infty}_{k-1}\frac{A^{y-1}}{y+s_{y}}dy.$$
Moreover, since $\frac{1}{k+s_{k}}\le \frac{1}{s_k}$ and  $\frac{1}{s_k}\int^{\infty}_{k-1}\frac{A^{y-1}}{s_{y}}dy$  decreases in $k$, we have
\begin{align*}
\E \tilde s_{0,1}^{2} &\le 2\sum_{k=1}^{\infty}\frac{1}{k+s_{k}}\sum_{n=k}^{\infty}\frac{A^{n-1}}{n+s_{n}} \\
&\le \frac{2}{1+s_{1}}\sum_{n=1}^{\infty} \frac{A^{n-1}}{n+s_{n}}+\frac{2}{2+s_{2}}\sum_{n=2}^{\infty}\frac{A^{n-1}}{s_{n}}+ 2\sum_{k=3}^{\infty}\frac{1}{s_{k}}\sum_{n=k}^{\infty}\frac{A^{n-1}}{s_{n}}\\
&\le \frac{2}{1+s_{1}}\sum_{n=1}^{\infty} \frac{A^{n-1}}{n+s_{n}} +\frac{2}{2+s_{2}}\sum_{n=2}^{\infty}\frac{A^{n-1}}{s_{n}} + 2\int_{2}^{\infty}\frac{1}{s_{x}} \int^{\infty}_{x-1}\frac{A^{y-1}}{s_{y}}dydx.
\end{align*}
Here, the reason to single out the first two terms is again to avoid integrating from 0 when applying an integral approximation in the next step. We note that $s_{1}=2(d-1)$ and $s_{2}\ge d$, thus the first two terms are smaller order of $(\log d/d)^{2}$, i.e., 
\begin{align*}
\frac{2}{1+s_{1}}\sum_{n=1}^{\infty} \frac{A^{n-1}}{n+s_{n}}&\le \frac{2}{(1+s_{1})^{2}}+\frac{2}{1+s_{1}}\sum_{n=2}^{\infty}\frac{A^{n-1}}{s_{n}}\\
&\le \frac{2}{(2d-1)^{2}}+ \frac{2}{2d-1} \cdot \frac{C\log d }{d} = o\big( (\log d/d)^{2}\big),\text{ and }\\
\frac{1}{2+s_{2}}\sum_{n=2}^{\infty}\frac{A^{n-1}}{s_{n}} &\le \frac{1}{d}\sum_{n=2}^{\infty}\frac{A^{n-1}}{s_{n}} \le \frac{C\log d}{d^{2}} =o\big( (\log d/d)^{2}\big),
\end{align*}
where we used the fact that $\sum_{n=2}^{\infty}\frac{A^{n-1}}{s_{n}}\le \frac{C\log d}{d}$; see the last term in \eqref{eqn:1st-sum-bound}. It remains to show that  for  $d$ sufficiently large 
\begin{align*}
2\int_{2}^{\infty}\frac{1}{s_{x}} \left(\int^{\infty}_{x-1}\frac{A^{y-1}}{s_{y}}dy\right)dx \le \left(\frac{\log d}{2d}\right)^{2} (1+o(1)).
\end{align*}
We break the integral into three parts:
\begin{align*}
\text{(I)}&:=2\int_{2}^{2d+1}\frac{1}{s_{x}} \left(\int^{2d}_{x-1}\frac{A^{y-1}}{s_{y}}dy\right)dx, \\
\text{(II)}&:=2\int_{2}^{2d+1}\frac{1}{s_{x}} \left(\int^{\infty}_{2d}\frac{A^{y-1}}{s_{y}}dy\right)dx, \\
\text{(III)}&:= 2\int_{2d+1}^{\infty}\frac{1}{s_{x}} \left(\int^{\infty}_{x-1}\frac{A^{y-1}}{s_{y}}dy\right)dx.
\end{align*}
For (III), when $x\ge 2d+1$ and $y\ge x-1\ge 2d$, both $s_{x}$ and $s_{y}$ are greater than $s_{2d} = 2(d-1)(2d)^{\frac{d-2}{d-1}}$. Thus, we replace them by $s_{2d}$ and evaluate the integrals to get
\begin{align*}
\text{(III)} &\le \frac{2}{s_{2d}^{2}}  \int_{2d+1}^{\infty}\int_{x-1}^{\infty}A^{y-1}dy dx = \frac{-2}{As_{2d}^{2} \cdot \log A}  \int_{2d+1}^{\infty} A^{x-1}dx\\
&= \frac{-2}{As_{2d}^{2} \cdot \log A}  \int_{2d}^{\infty} A^{x}dx = \frac{2A^{2d}}{As_{2d}^{2} \cdot (\log A)^{2}} .
\end{align*}
Note that  $2A^{2d}\to 2e^{-1}$ and $(\log A)^{2} \sim (-\frac{1}{2d})^{2}=\frac{1}{4d^{2}}$, thus we get for all $d$ sufficiently large, 
\begin{align*}
\text{(III)} &\le \frac{C_{1}\cdot 4d^{2}}{ [2(d-1)(2d)^{\frac{d-2}{d-1}}]^{2}}\sim  \frac{C_{2}d^{\frac{2}{d-1}}}{d^{2}}\sim \frac{C_{3}}{d^{2}} = o\big( (\log d/d)^{2}\big),
\end{align*}
where $C_{1}, C_{2}$ and $C_{3}$ are absolute constants that do not depend on $d$ and may vary from line to line.
For (II), again since $s_{y}\ge s_{2d}$ when $y\ge 2d$ and $s_{x}=(2d-1)x^{\frac{d-2}{d-1}}$, we have
\begin{align*}
\text{(II)} &\le \frac{2}{s_{2d}}  \int_{2}^{2d+1}\frac{1}{s_{x}}\int_{x-1}^{\infty}A^{y-1}dy dx = \frac{-2}{A^{2}s_{2d} \cdot \log A}  \int_{2}^{2d+1} \frac{A^{x}}{s_{x}}dx\\
&\le \frac{-2}{2(d-1)A^{2}s_{2d} \cdot \log A}  \int_{2}^{2d+1} \frac{A^{x}x^{\frac{1}{d-1}}}{x}dx\\
&\le \frac{- 2(2d+1)^{\frac{1}{d-1}}}{2(d-1)A^{2}s_{2d} \cdot \log A}  \int_{2}^{2d+1} \frac{A^{x}}{x}dx \\
&\le  \frac{- 2(2d+1)^{\frac{1}{d-1}}}{2(d-1)A^{2}s_{2d} \cdot \log A}  \int_{1}^{\infty} \frac{A^{x}}{x}dx.
\end{align*}
The rightmost integral is of order $\log d$ as $d\to\infty$, following \eqref{eqn:lhopital-1}. Thus for all sufficiently large $d$,
\begin{align*}
\text{(II)} &\le  \frac{- C_{1} (2d+1)^{\frac{1}{d-1}}}{2(d-1)A^{2}s_{2d} \cdot \log A}  \log d.
\end{align*}
As $d\to\infty$, $(2d+1)^{\frac{1}{d-1}}\to 1$, $2(d-1)\log A \to -1$ and $A\to 1$, thus for all $d$ large, 
\[
\text{(II)} \le  \frac{C_{2} \log d}{s_{2d}} = \frac{C_{3} d^{\frac{1}{d-1}}\log d}{d^{2}} =O(\log d/d^{2})= o\big( (\log d/d)^{2}\big).
\]
Finally, for (I), we plug in the expression for $s_{x}$ and $s_{y}$ and use the fact that $x, y$ are both less than $2d+1$, and
\begin{align*}
\text{(I)}&\le \frac{(2d+1)^{\frac{2}{d-1}}}{4A(d-1)^{2}}\int_{1}^{\infty}\frac{2}{x}\int_{x-1}^{\infty}\frac{A^{y}}{y}dy dx.
\end{align*}
The double integral is $\sim (\log d)^{2}$ as $d\to \infty$. To see this, we apply L'Hopital's rule and take the derivative with respect to $d$:
\begin{align*}
\lim_{d\to\infty} \frac{\int_{1}^{\infty}\frac{2}{x}\int_{x-1}^{\infty}\frac{A^{y}}{y}dy dx}{(\log d)^{2}} = \lim_{d\to\infty}\frac{\frac{1}{2d^{2}}\int_{1}^{\infty}\frac{1}{x}\int_{x-1}^{\infty}A^{y-1}dy dx}{(\log d)\cdot  \frac{1}{d}}\\
=\lim_{d\to\infty}\frac{\frac{-1}{2dA\log A}\int_{1}^{\infty}\frac{A^{x-1} }{x}dx}{\log d}   =\lim_{d\to\infty}\frac{\frac{-1}{2dA^{2}\log A}\int_{1}^{\infty}\frac{A^{x}}{x} dx}{\log d} =1,
\end{align*}
where the last equality follows from \eqref{eqn:lhopital-1}, $2d\log A \to -1$, and $A\to 1$ as $d\to \infty$. Plugging this back to the upper bound of (I), we obtain that for all sufficiently large $d$, 
\begin{align*}
\text{(I)}&\le \frac{(2d+1)^{\frac{2}{d-1}}}{4A(d-1)^{2}}(\log d)^{2} (1+o(1)) =\left(\frac{\log d}{2d}\right)^{2}(1+o(1)),
\end{align*} 
which concludes the proof for $\E \tilde s^{2}_{0,1} \le \left(\frac{\log d}{2d}\right)^{2} (1+o(1))$ in the Exponentially-weighted case.
\section{Generalization to edge-weight distribution $F$\label{sec:general}}
Recall that if $X$ is an Exponential(1) random variable, then $X/a$ follows an Exponential$(a)$ distribution. Thus if the passage times are i.i.d. Exponential$(a)$ distributed, we have 
\begin{align*}
\E \tilde s_{0,1} \sim \frac{\log d}{2ad}, \quad \E \tilde s^{2}_{0,1} \le \left(\frac{\log d}{2ad}\right)^{2} (1+o(1)),
\end{align*}
which implies $\text{Var}(\tilde s_{0,1})=o\big( (\log d/d)^{2}\big)$ and $\frac{2ad}{\log d} \tilde s_{0,1}\xrightarrow{P} 1$ as $d\to\infty$. In this section, we generalize this in-probability convergence to other distributions $F$ with a density $a>0$ at $x=0$, namely, $F$ satisfying
\begin{align}
\label{eqn:distributions}
\left| F(x)/x -a \right| \to 0, \text{ as } x\to0.
\end{align}
Note that the rate of convergence in \eqref{eqn:distributions} is irrelevant for the purpose of generalizing the in-probability convergence. Since we are working with two different probability distributions in this section, we will always write a superscript $F$ when we work with $F$-distributed passage times; the superscript is omitted when the passage times are Exponential$(a)$ distributed. We couple the distribution $F$ with an Exponential$(a)$ distribution using the left-continuous inverse function $F^{\ast}: [0,1]\to \mathbb R$ of $F$, i.e., for any $y\in [0,1]$
\[
F^{\ast}(y) := \inf\{x: F(x)\ge y\}.
\]
It follows that if $(\tau_{e})_{e\in \mathbb E}$ are i.i.d. Exponential$(a)$-distributed edge weights in a first-passage percolation model, then $(\tau_{e}^{F})_{e\in \mathbb E}:=(h(\tau_{e}))_{e\in \mathbb E}$ are i.i.d. $F$-distributed, where
\[
h(t) :=F^{\ast}(1-e^{-at})= \inf\{x\ge 0: F(x)\ge 1-e^{-at}\}.
\]
Note that $h$ is a monotonically increasing function. Also, since $F(x)\sim ax$ as $x\to 0$, then $F^{\ast}(y)\sim y/a$ as $y\to 0$. This implies $\lim_{t\to0}h(t)/t = 1$, i.e., for any $\epsilon > 0$, there is $\delta = \delta_{\epsilon} > 0$ such that 
\begin{align}
\label{eqn:h-near-0}
(1-\epsilon )t \le h(t) \le (1+\epsilon)t,\quad \text{ for all }t\in [0, \delta].
\end{align}
Denote by $\tilde s_{0,1}^{F}$ in the same way as $\tilde s_{0,1}$ but for the case when the passage times are i.i.d. $F$-distributed. Let $\Gamma$ be one path that realizes $\tilde s_{0,1}$ in the Exponential$(a)$ case. Then, for any $\delta >0$,
\begin{align}
\label{eqn:alledgesmall}
P(\tau_{e}\le \delta \text{ for all } e \in \Gamma) \to 1.
\end{align}
This is because the complement of the event satisfies
\[
P(\tau_{e} > \delta \text{ for some } e \in \Gamma) \le P(\tilde s_{0,1} > \delta) = P(2ad \tilde s_{0,1}/\log d > 2 \delta ad/\log d) \to 0,
\]
as we have already showed that $\frac{2ad}{\log d} \tilde s_{0,1} \xrightarrow{P} 1$ as $d\to\infty$. 

The main result of this section is the following. 

\begin{proposition}\label{conv:probaforS}
Assume \eqref{eqn:conditionat0} and \eqref{eqn:finite-expectation}. As $d$ goes to infinity, we have 
\begin{equation*}
\frac{2ad}{\log d} \tilde s_{0,1}^{F} \xrightarrow{P} 1.
\end{equation*}
\end{proposition}
\begin{proof}[Proof of Proposition \ref{conv:probaforS}]
We start with an upper bound for $\tilde s_{0,1}^{F}$.
Consider any $\eta > 0$. Choose $\epsilon \in (0, \eta)$ and $\delta>0$ according to $\epsilon$ such that \eqref{eqn:h-near-0} is satisfied. Then 
\begin{align*}
P\left(\frac{2ad}{\log d}\tilde s^{F}_{0,1} \le 1+\eta \right) &\ge P\left(\frac{2ad}{\log d}\tilde s^{F}_{0,1} \le 1+\eta,  \tau_{e}\le \delta \text{ for all } e \in \Gamma\right)\\
&\ge P\left(\frac{2ad}{\log d} T^{F}(\Gamma) \le 1+\eta,  \tau_{e}\le \delta \text{ for all } e \in \Gamma\right),
\end{align*}
where $T^{F}(\Gamma)$ denotes the passage time along the path $\Gamma$ with edge weights $(\tau_{e}^{F})_{e\in \mathbb E}:=(h(\tau_{e}))_{e\in \mathbb E}$
\[
T^{F}(\Gamma) := \sum_{e\in \Gamma}\tau_{e}^{F}= \sum_{e\in \Gamma}h(\tau_{e}).
\]
On the event that $\{\tau_{e}\le \delta \text{ for all } e \in \Gamma\}$, we have $h(\tau_{e})\le (1+\epsilon)\tau_{e}$, and thus $\sum_{e\in \Gamma}h(\tau_{e})\le (1+\epsilon)\sum_{e\in \Gamma}\tau_{e} = (1+\epsilon)\tilde s_{0,1}$. Thus, the probability above is bounded from below by
\begin{align*}
P\left(\frac{2ad}{\log d}\tilde s^{F}_{0,1} \le 1+\eta \right)&\ge
P\left(\frac{2ad}{\log d}\tilde s_{0,1}\le \frac{1+\eta}{1+\epsilon},  \tau_{e}\le \delta \text{ for all } e \in \Gamma\right) \to 1,
\end{align*}
due to $\frac{1+\eta}{1+\epsilon} > 1$, $\frac{2ad}{\log d} \tilde s_{0,1}\xrightarrow{P} 1$ and \eqref{eqn:alledgesmall}. The lower bound is similar. Let $\Gamma^{F}$ be one path that realizes $\tilde s_{0,1}^{F}$ in the case that edge weights are i.i.d. $F$-distributed. Take any $\eta > 0$ and $\epsilon \in (0, \eta)$, and choose $\delta=\delta_{\epsilon}$ according to $\epsilon$ such that \eqref{eqn:h-near-0} is satisfied. Then 
$\tilde s_{0,1}^{F}=\sum_{e\in \Gamma^{F}}\tau_{e}^{F} = \sum_{e\in \Gamma^{F}}h(\tau_{e})$
\begin{align*}
&P\left(\frac{2ad}{\log d}\tilde s^{F}_{0,1} \le 1-\eta \right) =P\left(\frac{2ad}{\log d}\sum_{e\in \Gamma^{F}}h(\tau_{e})\le 1-\eta \right)\\
&=P\left(\frac{2ad}{\log d}\sum_{e\in \Gamma^{F}}h(\tau_{e})\le 1-\eta,\ \tau_{e}\le \delta \text{ for all } e \in \Gamma^{F} \right)\\
&\qquad + P\left(\frac{2ad}{\log d}\sum_{e\in \Gamma^{F}}h(\tau_{e})\le 1-\eta,\ \tau_{e}> \delta \text{ for some } e \in \Gamma^{F} \right)\\
&\le P\left(\frac{2ad}{\log d}\sum_{e\in \Gamma^{F}}(1-\epsilon)\tau_{e}\le 1-\eta,\ \tau_{e}\le \delta \text{ for all } e \in \Gamma^{F} \right) \\
&\qquad + P\left(\frac{2ad}{\log d}h(\delta)\le 1-\eta,\ \tau_{e}> \delta \text{ for some } e \in \Gamma^{F} \right)\\
&\le P\left(\frac{2ad}{\log d}\tilde s_{0,1}\le \frac{1-\eta}{1-\epsilon} \right) +  P\left(\frac{2ad}{\log d}(1-\epsilon)\delta \le 1-\eta \right).
\end{align*}
The first term vanishes as $d\to \infty$ because $\frac{2ad}{\log d} \tilde s_{0,1} \xrightarrow{P} 1$  and  $\frac{1-\eta}{1-\epsilon}<1$; and the second term is 0 when $d$ is sufficiently large.  This complete the proof for $\frac{2ad}{\log d} \tilde s_{0,1}^{F} \xrightarrow{P}1$.
\end{proof}

\section{Proofs of Theorems \ref{thm:mainUB}\label{sec:ui} and \ref{thm-convPL1}}
In this section, we prove Theorem \ref{thm:mainUB} and Theorem \ref{thm-convPL1}. 
For this, we show the following proposition.
\begin{proposition}\label{prop:lastone} Assume \eqref{eqn:conditionat0} and \eqref{eqn:finite-expectation}. Then, as $d \to \infty$,
\begin{equation*}\frac{2ad}{\log d}\E \tilde s_{0,1}^{F}\to 1.
\end{equation*}
\end{proposition}

\begin{proof}[Proof of Proposition \ref{prop:lastone}]
In view of Proposition \ref{conv:probaforS}, it suffices to show that the collection of random variables $\{X_{d}\}_{d\ge 1}:=\{\frac{2ad}{\log d}\tilde s_{0,1}^{F}\}_{d\ge 1}$ is uniformly integrable, i.e.,
\begin{align}
\label{eqn:ui-for-s01}
\lim_{M\to \infty}\sup_{d}\E \left[X_{d}\mathbbm 1_{\{X_{d}\ge M\}}\right] = 0.
\end{align}
To do this, we adopt the ``search-and-cross'' strategy that were used in \cite{KES86AFPP} and \cite{AT16highdimenFPP} when estimating $\mathbb E\tilde s_{0,1}^{F}$, which we briefly describe below. In order to get to $\mathbb H_{1}$ from $\mathbf 0$ quickly, one can first make a move in one of the $e_{p+2}, \ldots, e_{d}$ directions and then search for a fast path $\gamma$ (of length $n$) in a subspace of $\mathbb Z^{d}$ spanned by $\{\pm e_{2}, \ldots, \pm e_{p+1}\}$ that ends with the last step in the $e_{1}$ direction leading to $\mathbb H_{1}$. For $j=p+2, \ldots, d$, if the first step has passage time $\tau_{e_{j}}^{F} \le y$ and the path $T^{F}(\gamma)\le x$ (denote this event by $F_{j}$), follow $e_{j}$ and then this path $\gamma$ to get to $\mathbb H_{1}$; otherwise move directly from $\mathbf 0$ to $\mathbb H_{1}$ using the edge in the $e_{1}$ direction.  Thus $\tilde s_{0,1}^{F}$ is bounded from above by
\[
\tilde s_{0,1}^{F} \le (x+y)\mathbbm 1_{\cup_{j=p+2, \ldots, d}^{\infty}F_{j}} + \mathbbm 1_{\cap_{j=p+2, \ldots, d}^{\infty}F_{j}^{c}} \tau_{e_{1}}^{F}
\]
Note that the choices for $x, y$, the length $n$ of the fast path $\gamma$, and the dimension $p$ of the search space are different in  \cite{KES86AFPP} and \cite{AT16highdimenFPP}. For example, in \cite{KES86AFPP}, Kesten chose
$$
p =\left \lfloor \frac{d}{2}\right\rfloor, \quad n = \left \lfloor \frac{3}{4}\log d\right\rfloor, \quad x = \frac{9\log d}{4ad},
$$
and with these choices, he showed that the probability of finding such a fast path $\gamma$ is at least $\frac{1}{4}$, based on a second-moment method. The argument relies on a convergence rate of order $O(|\log x|^{-1})$ as shown in \eqref{eqn:conditionat0} for the edge weight-distribution $F$, which ensured that there were sufficiently many fast paths as $d\to \infty$; see \cite[pp.245]{KES86AFPP} for the details on the moment computation. Here we adopt Kesten's choices for $p, n$ and $x$, but other choices can also work. Moreover, for small enough $y$,  the probability $P(\tau_{e_{j}}^{F}\le y) \sim ay$ due to \eqref{eqn:distributions}. Thus, by choosing $y=32\log d/(ad)$, we have that for all $d$ sufficiently large (say, $d > d_{0}$),
\[
P(\tau_{e_{j}}^{F}\le y) \ge \frac{ay}{2} = \frac{16 \log d}{d},
\]
which leads to 
\[
P(F_{j}) \ge \frac{1}{4} \cdot P(\tau_{e_{j}}^{F}\le y)  \ge \frac{4\log d}{d}
\]
because for $j=p+2, \ldots, d$, the random variable $\tau_{e_{j}}^{F}$ is independent of the edge weights along the fast path $\gamma$. It then follows that
\begin{align*}
\E \left[X_{d}\mathbbm 1_{\{X_{d}\ge M\}}\right] &\le \frac{2ad}{\log d} \E \left[(x+y)\mathbbm 1_{\cup_{j=p+2, \ldots, d}^{\infty}F_{j}}\mathbbm 1_{\{\tilde s_{0,1}^{F}\ge M\log d/(2ad)\}}\right] \\
&\quad + \frac{2ad}{\log d}\E \left[\mathbbm 1_{\cap_{j=p+2, \ldots, d}^{\infty}F_{j}^{c}}\mathbbm 1_{\{\tilde s_{0,1}^{F}\ge M\log d/(2ad)\}}\right] \E ( \tau_{e_{1}}^{F}).
\end{align*}
If we choose $M\ge 100$, then the event $\tilde s_{0,1}^{F} > \frac{M\log d }{ 2ad} > x+y$ implies that none of the $F_{j}$ events would happen, and the first term above is zero. Thus, for $M\ge 100$, we have
\begin{align*}
\E \left[X_{d}\mathbbm 1_{\{X_{d}\ge M\}}\right] &\le \frac{2ad}{\log d}\E \left[\mathbbm 1_{\cap_{j=p+2, \ldots, d}^{\infty}F_{j}^{c}}\right] \E ( \tau_{e_{1}}^{F})\\
&\le \frac{2ad}{\log d}\left(1-\frac{4\log d}{d}\right)^{d-p-1}\E ( \tau_{e_{1}}^{F})\\
&\le \frac{2ad}{\log d}\left(1-\frac{4\log d}{d}\right)^{\frac{d}{3}}\E ( \tau_{e_{1}}^{F}) \le \frac{2ad}{\log d}\cdot C d^{-\frac{4}{3}} \to 0, \quad \text{as } d\to\infty.
\end{align*}
Therefore, for any $\epsilon > 0$, choose $d_{1}$ such that $\E \left[X_{d}\mathbbm 1_{\{X_{d}\ge 100\}}\right] \le \epsilon/2$ for all $d > d_{1}$. Then for $M\ge 100$,
\begin{align*}
\sup_d \E[X_d \mathbbm{1}_{\{X_d \geq M\}}] &\le \sup_{d \le d_1} \E[X_d \mathbbm{1}_{\{X_d \geq M\}}] + \sup_{d > d_1} \E[X_d \mathbbm{1}_{\{X_d \geq 100\}}] \\
&\le \sup_{d \le d_1} \E[X_d \mathbbm{1}_{\{X_d \geq M\}}] + \frac{\epsilon}{2}.
\end{align*}
The first term vanishes by sending $M\to \infty$. This finishes the proof of uniform integrability \eqref{eqn:ui-for-s01}. 
\end{proof}

The proof of our two main theorems are now straight-forward.

\begin{proof}[Proof of Theorem \ref{thm:mainUB}]
It follows by combining Propositions \ref{prop:lastone} with the fact that $\mu_{d}^{F}(e_{1})\le \E \tilde s_{0,1}^{F}$, see \cite[pp.246]{KES86AFPP} or \cite[Lemma 5.2]{SW78}. 
\end{proof}

\begin{proof}[Proof of Theorem \ref{thm-convPL1}]
The proof follows by combining Proposition \ref{conv:probaforS} and the uniform integrability  \eqref{eqn:ui-for-s01}. 
\end{proof}

\end{document}